\documentclass[12pt,reqno]{amsart}
\usepackage[english]{babel}
\usepackage{amsthm}
\usepackage{inputenc}
\usepackage{amssymb}
\usepackage{amsmath,amscd,amsfonts}
\newtheorem{thm}{Theorem}[section]
\newtheorem{cor}[thm]{Corollary}
\newtheorem{lm}[thm]{Lemma}
\newtheorem{pr}[thm]{Proposition}
\newtheorem{defn}[thm]{Definition}
\theoremstyle{remark}
\newtheorem{rem}[thm]{Remark}
\newtheorem{ex}[thm]{Example}
\def\Ann{\operatorname{Ann}}
\def\ann{\operatorname{ann}}
\def\Spec{\operatorname{Spec}}
\numberwithin{equation}{section}

\title{Automorphisms and derivations of Leibniz algebras}

\author[M. Ladra \and  I. M. Rikhsiboev \and R. M. Turdibaev]{M. Ladra \and  I. M. Rikhsiboev \and R. M. Turdibaev}
\begin{document}

\begin{abstract}
The present work is devoted to the extension of some general properties of automorphisms and derivations which are known for Lie algebras to finite dimensional complex Leibniz algebras. The analogues of the Jordan-–Chevalley decomposition for derivations and the multiplicative decomposition for automorphisms of finite dimensional complex Leibniz algebras are obtained.
\end{abstract}

\keywords{Leibniz algebra; Lie algebra; derivation; automorphism; nilpotent Leibniz algebra}
\subjclass[2010]{17A32, 17A36, 17B40.}

\maketitle

\section{Introduction}

 Leibniz algebras were first introduced by Loday in \cite{Lo1,Lo2} as a non-antisymmetric version
 of Lie algebras. Many results of Lie algebras were also established in Leibniz algebras. Since the study
  of the properties of derivations and automorphisms of Lie algebra play an essential role in the
  theory of Lie algebras,  the question naturally arises whether the corresponding results can be extended to the more general framework of the Leibniz algebras.

In this work we consider some general properties of derivations and automorphisms of Leibniz algebras. We extend some results obtained for derivations and automorphisms of Lie algebras in \cite{Gant,Jac1} to the case
of Leibniz algebras. Among them we prove the analogue of the Jordan-–Chevalley decomposition,
 which expresses a derivation of a Leibniz algebra as the sum of its commuting semisimple and nilpotent parts. Similar results were
 established in \cite{Gant} and \cite{Hum} for Lie algebras. If the
linear operator is invertible, then the Jordan--Chevalley decomposition expresses it as a product of commuting semisimple and unipotent operators.
 Gantmacher \cite{Gant}, in the theory of Lie algebras,  proved that any automorphism of Lie algebras decomposes into the product of commuting semisimple automorphism and exponent of a nilpotent derivation. In this work we verify that the same results hold in Leibniz algebras.

In 1955, Jacobson \cite{Jac1} proved that every Lie algebra over a field of characteristic zero
admitting a nonsingular derivation is nilpotent. The problem whether the converse of this statement
 is correct remained open until that an example of a nilpotent Lie algebra in which every derivation
 is nilpotent (and hence, singular) was constructed in \cite{Dix}. Nilpotent Lie
  algebras with this property were named characteristically nilpotent Lie algebras.
  In \cite{Hakim} it was proved that every irreducible component of the variety of
  complex filiform Lie algebras of dimension greater than 7 contains a Zariski open set
  consisting of characteristically nilpotent Lie algebras. Note that among the nilpotent
   Lie algebras of dimension less than 7, characteristically nilpotent Lie algebras do not occur due to the classification given in \cite{Goze}.

In this paper we prove that a finite dimensional complex Leibniz algebra admitting a non-degenerate derivation is nilpotent. Similar to the Lie case, the inverse of this statement does not hold. The notion of characteristically nilpotent Leibniz algebra is defined similarly as in the Lie case. It was established in \cite{Omirov} that characteristically nilpotent non-Lie filiform Leibniz algebras occur starting with dimension 5.

An advance of Theorems \ref{nilpotent} and  \ref{aut_prime} was presented in \cite{Rik}. We also have found in the literature the reference \cite{BHMSS}, where some results here proved are showed with another techniques.

In the present paper, all vector spaces and algebras are considered over the field of the complex numbers  $\mathbb{C}$.
 We will denote by  $C^i_{k}$  the binomial coefficient $\binom{k}{i}$.

\section{Preliminaries}

In this section we present some known notions and results concerning Leibniz algebras that we use further in this work.

\begin{defn} An algebra $L$ over a field $F$ is called a Leibniz algebra if for any $x,y,z\in L$ the Leibniz identity
\[[[x,y],z]=[[x,z],y]+[x,[y,z]]\] is satisfied, where $[-,-]$ is the multiplication in $L$.
\end{defn}

In other words, the right multiplication operator  $[- ,z ]$ by any element $z$ is a derivation (see \cite{Lo1}).

 Any Lie algebra is a Leibniz algebra, and conversely any Leibniz algebra $L$ is a Lie algebra if $[x,x]=0$ for all $x \in L$.
 Moreover, if $L^{\ann}= \mbox{ ideal} <[x,x] \, | \,  x \in L>$, then the factor algebra $L/L^{\ann}$ is a Lie algebra.

For a Leibniz algebra $L$ consider the following derived and lower central series:
\begin{align*}
  & \mbox{(i)}   &L^{(1)}= & \ L, \  & L^{(n+1)}= & \ [L^{(n)},L^{(n)}],  & n>1; \\
  & \mbox{(ii)}  & L^1= &  \ L, \ & L^{n+1}=&  \ [L^n,L],  & n>1.
\end{align*}

\begin{defn} An algebra $L$ is called solvable (nilpotent) if there exists $s\in \mathbb{N}$ \ ($k\in \mathbb{N}$, respectively) such that $L^{(s)}=0$ \ ($L^k=0$, respectively).
\end{defn}

The following theorem from linear algebra characterizes the decomposition of a vector space into a direct sum of characteristic subspaces.

\begin{thm}[\cite{Mal}] Let $A$ be a linear transformation of vector space $V$. Then $V$ decomposes into the direct sum of characteristic subspaces $V=V_{\lambda_1}\oplus V_{\lambda_2} \oplus \dots \oplus V_{\lambda_k}$  with respect to $A$, where $V_{\lambda_i}=\{x\in V\ | \ (A-\lambda_i I)^k(x)=0 \mbox{ for some } k\in \mathbb{N}\}$ and $\lambda_i, 1\leq i \leq k$, are eigenvalues of $A$.
\end{thm}

The following proposition gives the (additive) Jordan--Chevalley decomposition of an endomorphism.

\begin{pr}[\cite{Hum}]\label{Hump}
Let $V$ be a finite dimensional vector space over $\mathbb{C}$,\ $x\in End(V)$.
\begin{itemize}
  \item[(i)] There exist unique $x_d, x_n\in End(V)$ satisfying the conditions: $x=x_d+x_n$, $x_d$ is diagonalizable, $x_n$ is nilpotent, $x_d$ and $x_n$ commute.
  \item[(ii)] There exist polynomials $p(t),\ q(t) \in \mathbb{C}[t]$, without constant term, such that $x_d=p(x),\ x_n=q(x)$. In particular, $x_d$ and $x_n$ commute with any endomorphism commuting with $x$.
  \item[(iii)] If $A\subseteq B \subseteq V$ are subspaces and $x$ maps $B$ in $A$, then $x_d$ and $x_n$ also map $B$ in $A$.
\end{itemize}
\end{pr}

In Leibniz algebras a derivation is defined as usual.

\begin{defn} A linear operator $d \colon L\to L$ is called a derivation of $L$ if
\[d([x,y])=[d(x),y]+[x,d(y)]  \  \  \mbox{ for any } x,y\in L.\]
\end{defn}

For an arbitrary element $x\in L$, we consider the right multiplication operator $R_x \colon L\to L$
defined by $R_x(z)=[z,x]$. Right multiplication operators are derivations of the algebra $L$.
 The set $R(L)=\{R_x\ |\ x\in L\}$ is a Lie algebra with respect to the commutator and the following identity holds:
\begin{equation}\label{2.1}
R_xR_y-R_yR_x=R_{[y,x]} \, .
\end{equation}

\begin{defn}[\cite{Jac}]
 A subset $S$ of an associative algebra $A$ over a field $F$ is called weakly closed if for every pair $(a,b)\in S\times S$ an element  $\gamma(a,b)\in F$
  is defined  such that $ab+\gamma(a,b)\, ba\in S$.
\end{defn}

Further, we need a result concerning the weakly closed sets.

\begin{thm}[\cite{Jac}] \label{closed} Let $S$ be a weakly closed subset of the associative algebra $A$ of linear transformations of a finite-dimensional vector space $V$ over $F$. Assume every $W\in S$ is nilpotent, that is, $W^k=0$ for some positive integer $k$. Then the enveloping associative algebra $S^*$ of $S$ is nilpotent.
\end{thm}

The classical Engel's theorem from Lie algebras has the following analogue in Leibniz algebras.

\begin{thm}[\cite{Ayup}, Engel's theorem] \label{Engel} A Leibniz algebra $L$ is nilpotent if and only if $R_x$ is nilpotent for any $x\in L$.
\end{thm}

\begin{defn} The set $\Ann_R(L)=\{x\in L \, | \, [L,x]=0\}$ of a Leibniz algebra $L$ is called the right annihilator of $L$.
\end{defn}
One can show that $\Ann_R(L)$ is an ideal of $L$.

For a Leibniz algebra $L$, let $H$ be a maximal solvable ideal in the sense that $H$ contains any solvable ideal of $L$. Since the sum of solvable ideals is again a solvable ideal (see \cite{Ayup}), this implies the existence of a unique maximal solvable ideal, which is said to be the \emph{radical} of $L$.

Similarly, let $K$ be a maximal nilpotent ideal of Leibniz algebra $L$. Since the sum of nilpotent
 ideals is a nilpotent ideal (see \cite{Ayup}), this implies the existence of a unique maximal nilpotent ideal,
  which is said to be the \emph{nilradical} of $L$.  Notice that, the nilradical does not possess the properties of the radical in the sense of Kurosh.

\section{Main Result}

This section is devoted to the extension of known results for Lie algebras on automorphisms and derivations to Leibniz algebras.

\begin{lm}\label{der} Let $L$ be a finite dimensional Leibniz algebra with a derivation $d$ defined on it and $L=L_{\rho_1}\oplus \dots \oplus L_{\rho_s}$ be a decomposition of $L$ into characteristic spaces with respect to $d$. Then for any $\alpha,\beta\in \Spec(d)$ we have
\[
[L_{\alpha},L_{\beta}]\subseteq\left\{\begin{array}{cl}
L_{\alpha+\beta}& \textrm{ if } \alpha+\beta \textrm{ is an eigenvalue of } d\\
0 & \textrm{ if } \alpha+\beta \textrm{ is not an eigenvalue of } d \ .\\
\end{array}\right.
\]
\end{lm}

\begin{proof}
First observe that $(d-(\alpha+\beta)I)([x,y])=[d(x),y]+[x,d(y)]-(\alpha+\beta)[x,y]= [(d-\alpha I)(x),y]+[x,(d-\beta I)(y)]$. Now assume that
\begin{equation} \label{3.1}
(d-(\alpha+\beta)I)^k([x,y])=\sum_{i=0}^k C^i_k [(d-\alpha I)^i(x),(d-\beta I)^{k-i}(y)]
\end{equation}
 for some $k>1$. Then
\begin{multline*}
(d-(\alpha+\beta)I)^{k+1}([x,y])=(d-(\alpha+\beta)I)\left(\sum_{i=0}^k C^i_k [(d-\alpha I)^i(x),(d-\beta I)^{k-i}(y)]\right) \\
=\sum_{i=0}^k C^i_k [(d-\alpha I)^{i+1}(x),(d-\beta I)^{k-i}(y)]+\sum_{i=0}^k C^i_k [(d-\alpha I)^i(x),(d-\beta I)^{k-i+1}(y)]\\
=[(d-\alpha I)^{k+1}(x),(y)]+\sum_{i=0}^{k-1} C^i_k [(d-\alpha I)^{i+1}(x),(d-\beta I)^{k+1-(i+1)}(y)]\\
+\sum_{i=1}^{k} C^i_k [(d-\alpha I)^{i}(x),(d-\beta I)^{k+1-i}(y)]+[x,(d-\beta I)^{k+1}(y)]\\
=[(d-\alpha I)^{k+1}(x),y]+\sum_{i=1}^{k} (C^{i-1}_k+C^i_k) [(d-\alpha I)^{i}(x),(d-\beta I)^{k+1-i}(y)]+[x,(d-\beta I)^{k+1}(y)]\\
=[(d-\alpha I)^{k+1}(x),y]+\sum_{i=1}^{k} C^{i}_{k+1} [(d-\alpha I)^{i}(x),(d-\beta I)^{k+1-i}(y)]+[x,(d-\beta I)^{k+1}(y)]\\
=\sum_{i=0}^{k+1} C^i_{k+1} [(d-\alpha I)^i(x),(d-\beta I)^{k+1-i}(y)] \, .
\end{multline*}
Hence \eqref{3.1} holds for any $k\in \mathbb{N}$.

Consider $x\in L_{\alpha},\, y\in L_{\beta}$. Then there exist natural numbers $p,q$ such that $(d-\alpha I)^p(x)=0$ and $(d-\beta I)^q(y)=0$.
In \eqref{3.1} taking $k=p+q$ we have that $\big(d-(\alpha+\beta)I\big)^k([x,y])=0$ which completes the proof of the statement of the lemma.
\end{proof}

Let $d$ be a derivation of a Leibniz algebra $L$. From the definition of derivation it is straightforward that $\ker d$ is a subalgebra.
 Moreover, by Lemma \ref{der} we have $[L_0,L_0]\subseteq L_0$ and hence $L_0$ is also a subalgebra of $L$.

The following theorem is a generalization of the analogous result in the theory of Lie algebras established in \cite{Gant}.

\begin{thm}\label{D=D_0+T} Let $D$ be a derivation of a Leibniz algebra $L$. Then there exists a unique diagonalizable derivation $D_0$ and a unique nilpotent derivation $T$ such that $D=D_0+T$ and $D_0 T=TD_0$.
\end{thm}

\begin{proof} Let $L=L_{\rho_1}\oplus \dots \oplus L_{\rho_s}$ be a decomposition of $L$ into characteristic spaces with respect to $d$. Let us define a linear operator $D_0 \colon L\to L$ as $D_0(x)=\rho_i x$ for $x\in L_{\rho_i}$. Then $D_0$ is obviously diagonalizable and $D_0 D=DD_0$.

Now we  show that $D_0$ is a derivation of $L$.

By Lemma \ref{der} if $x\in L_{\rho_i}, y\in L_{\rho_j}$ we obtain $[x,y]\in L_{\rho_i+\rho_j}$ if $\rho_i +\rho_j $ is an eigenvalue and $[x,y]=0$ otherwise.
If $\rho_i+\rho_j$ is an eigenvalue of $D$, then we obtain
\begin{align*}
D_0([x,y])=& \ (\rho_i+\rho_j)[x,y]\, ,\\
[D_0(x),y]+[x,D_0(y)]=& \ [\rho_ix,y]+[x,\rho_j y]=(\rho_i+\rho_j)[x,y] \, .
\end{align*}
So $D_0([x,y])=[D_0(x),y]+[x,D_0(y)]$.

If $\rho_i+\rho_j$ is not an eigenvalue, then $[x,y]=0$ and again we obtain  $D_0([x,y])=0$ and $[D_0(x),y]+[x,D_0(y)]=(\rho_i +\rho_j)[x,y]=0$. Hence, $D_0$ is a derivation.

Now denote by $T=D-D_0$. Obviously, $T$ is a derivation of $L$ and $T$ is nilpotent. Moreover, $T$ commutes with $D_0$.

The uniqueness of such decomposition follows from Proposition \ref{Hump}.
\end{proof}

In order to obtain a similar result for automorphisms of Leibniz algebras we need the following lemma.
\begin{lm}\label{P^k} Let $P$ be a nilpotent transformation of Leibniz algebra $L$ such that $P+I$ is an automorphism. Then
\begin{equation}\label{3.2}
P^k([x,y])=\sum_{i=0}^k \sum_{j=0}^i C_k^i C_i^j [P^{k-j}(x),P^{k-i+j}(y)]
\end{equation}
 for all $k\in \mathbb{N}$.
\end{lm}

\begin{proof} Let us denote $Q=P+I$. Since $Q$ is an automorphism we obtain
\begin{multline*}
P([x,y])=(Q-I)([x,y])=[Q(x),Q(y)]-[x,y] \\
= [Q(x)-x,Q(y)-y]+[Q(x)-x,y]+[x,Q(y)-y]\\
=[P(x),P(y)]+([P(x),y]+[x,P(y)])=
\sum_{i=0}^1 \sum_{j=0}^i C_1^i C_i^j [P^{1-j}(x),P^{1-i+j}(y)] \, .
\end{multline*}

Now assume that \eqref{3.2} holds for some natural $k>1$.
Then
\begin{multline*}
P^{k+1}([x,y])=\sum_{i=0}^k \sum_{j=0}^i C_k^i C_i^j P([P^{k-j}(x),P^{k-i+j}(y)])\\
=\sum_{i=0}^k C_k^i \sum_{j=0}^i C_i^j \left([P^{k-j+1}(x),P^{k-i+j+1}(y)]+ [P^{k-j+1}(x),P^{k-i+j}(y)]+[P^{k-j}(x),P^{k-i+j+1}(y)]\right)\, .
\end{multline*}

Consider
\begin{multline*}
\sum_{j=0}^i C_i^j [P^{k-j+1}(x),P^{k-i+j}(y)]+\sum_{j=0}^i C_i^j[P^{k-j}(x),P^{k-i+j+1}(y)]\\
=C_i^0 [P^{k+1}(x),P^{k-i}(y)]+\sum_{j=1}^i\big( C_i^j[P^{k+1-j}(x),P^{k-i+j}(y)] \\
+C_i^{j-1}[P^{k+1-j}(x),P^{k-i+j}(y)]\big) +C^i_i[P^{k-i}(x),P^{k+1}(y)] \\=C_{i+1}^0 [P^{k+1}(x),P^{k+1-(i+1)}(y)]\\
+\sum_{j=1}^i\left( C_i^j+C_i^{j-1}\right)[P^{k+1-j}(x),P^{k+1-(i+1)+j}(y)]+C^{i+1}_{i+1}[P^{k+1-(i+1)}(x),P^{k+1}(y)]\,.
\end{multline*}
Using the fact $C_i^j+C_i^{j-1}=C_{i+1}^{j}$ we obtain
\begin{multline*}
\sum_{j=0}^i C_i^j [P^{k-j+1}(x),P^{k-i+j}(y)]+\sum_{j=0}^i C_i^j[P^{k-j}(x),P^{k-i+j+1}(y)]\\
=\sum_{j=0}^{i+1}C_{i+1}^j[P^{k+1-j}(x),P^{k+1-(i+1)+j}(y)] \, .
\end{multline*}
Now
\begin{multline*}
P^{k+1}([x,y])=\sum_{i=0}^k \sum_{j=0}^i C_k^i C_i^j [P^{k-j+1}(x),P^{k-i+j+1}(y)]\\
+\sum_{i=0}^k \sum_{j=0}^{i+1}C_k^i C_{i+1}^j[P^{k+1-j}(x),P^{k+1-(i+1)+j}(y)]\\
=[P^{k+1}(x),P^{k+1}(y)]+\sum_{i=0}^{k-1} \sum_{j=0}^{i+1} C_k^{i+1} C_{i+1}^j [P^{k-j+1}(x),P^{k+1-(i+1)+j}(y)]\\
+\sum_{i=0}^{k-1}\sum_{j=0}^{i+1} C_k^i C_{i+1}^j[P^{k+1-j}(x),P^{k+1-(i+1)+j}(y)]+ \sum_{j=0}^{k+1} C_{k+1}^j[P^{k+1-j}(x),P^{j}(y)]\\
=[P^{k+1}(x),P^{k+1}(y)]+\sum_{i=0}^{k-1} \sum_{j=0}^{i+1} \left(C_k^{i+1}+C_k^i\right) C_{i+1}^j [P^{k-j+1}(x),P^{k+1-(i+1)+j}(y)]\\
+\sum_{j=0}^{k+1} C_{k+1}^j[P^{k+1-j}(x),P^{j}(y)]=[P^{k+1}(x),P^{k+1}(y)]\\
+\sum_{i=0}^{k-1} \sum_{j=0}^{i+1}C_{k+1}^{i+1} C_{i+1}^j [P^{k-j+1}(x),P^{k+1-(i+1)+j}(y)]+
\sum_{j=0}^{k+1} C_{k+1}^j[P^{k+1-j}(x),P^{j}(y)]\\
=[P^{k+1}(x),P^{k+1}(y)]+\sum_{i=1}^{k} \sum_{j=0}^{i}C_{k+1}^{i} C_{i}^j [P^{k-j+1}(x),P^{k+1-i+j}(y)]\\
+\sum_{j=0}^{k+1} C_{k+1}^j[P^{k+1-j}(x),P^{j}(y)]=
\sum_{i=0}^{k+1} \sum_{j=0}^{i}C_{k+1}^{i} C_{i}^j [P^{k-j+1}(x),P^{k+1-i+j}(y)]\, .
\end{multline*}

Thus, \eqref{3.2} is proved.
\end{proof}

The next lemma presents the similar result for automorphisms of Leibniz algebras as Lemma \ref{der}  does for derivations. Notice that, it also generalizes the result for Lie algebras given in \cite{Gant}.

\begin{lm}\label{aut} Let $L$ be a finite dimensional Leibniz algebra and $L=L_{\rho_1}\oplus \dots \oplus L_{\rho_s}$ be a decomposition of $L$ into characteristic spaces with respect to an automorphism $A$. Then for any $\alpha,\beta\in \Spec(A)$ we have
\[[L_{\alpha},L_{\beta}]\subseteq\left\{\begin{array}{cl}
L_{\alpha\beta}& \textrm{ if } \alpha\beta \textrm{ is an eigenvalue of } A \\
0 & \textrm{ if } \alpha\beta \textrm{ is not an eigenvalue of } A \ . \\
\end{array}\right.  \]
\end{lm}

\begin{proof} First observe that
\begin{multline*}
(A-\alpha\beta I)([x,y])=[A(x),A(y)]-\alpha\beta[x,y]\\
=[(A-\alpha I)(x),(A-\beta I)(y)]+[(A-\alpha I)(x),\beta y]+[\alpha x,(A-\beta I)(y)] \,.
\end{multline*}
Similarly to the proof of Lemma \ref{P^k} one can establish by induction
\begin{equation}\label{3.3}
(A-\alpha\beta I)^k([x,y])=\sum_{j=0}^k\sum_{i=0}^j \alpha^i \beta^{j-i} C^j_k C^i_j[(A-\alpha I)^{k-i}(x),(A-\beta I)^{k-j+i}(y)] \, .
\end{equation}

Now let $x\in L_{\alpha}$ and $y\in L_{\beta}$. Then there exist natural numbers $p,q$ such that $(A-\alpha I)^p(x)=0$ and $(A-\beta I)^q(y)=0$.
In \eqref{3.3} taking $k=p+q$ we have that  $(A-\alpha\beta I)^k([x,y])=0$ which completes the proof of the lemma.
\end{proof}

Below, we establish a technical lemma and a corollary in order to obtain a  similar result to Theorem \ref{D=D_0+T} for automorphisms of Leibniz algebra.

\begin{lm}\label{polynom} For any polynomial $P$ of degree less than $n$, where $n \in
\mathbb{N}$, the following equality holds:
\[\sum_{i =0}^{n}(-1)^i C^{i}_{n} P(i) = 0 \,.\]
\end{lm}

\begin{proof} Since $\deg P(x)<n$, applying Lagrange interpolation
formula to the points $x_k=k,\, 0\leq k \leq n-1$ we obtain
$P(x)=\displaystyle \sum_{k=0}^{n-1} q_k(x) P(k)$, where $q_k(x)=\displaystyle \frac{x(x-1)\cdots \big(x-(k-1)\big) \cdot \big(x-(k+1)\big) \cdots
\big(x-(n-1)\big)}{k(k-1)\cdots 1\cdot(-1)(-2)\cdots\big(-(n-1-k)\big)}$.

Now
\begin{multline*}
q_k(n)=\displaystyle\frac{n(n-1)\cdots\big(n-(k-1)\big)\cdot\big(n-(k+1)\big)\cdots
\big(n-(n-1)\big)}{k(k-1)\cdots 1\cdot(-1)(-2)\cdots\big(-(n-1-k)\big)}\\
=\frac{n!}{(-1)^{n-1-k} k! (n-k)!}=\frac1{(-1)^{n-1}}\cdot(-1)^{k}C^k_n \,.
\end{multline*}

Thus $P(n)=\displaystyle\sum_{k=0}^{n-1}q_k(n)P(k)=\frac1{(-1)^{n-1}}\sum_{k=0}^{n-1}(-1)^{k}C^k_n\cdot
P(k)$.

Hence, $0=\displaystyle \sum_{k=0}^{n-1}(-1)^{k}C^k_n\cdot
P(k)+(-1)^nC^n_nP(n)=\displaystyle\sum_{i = 0}^{n}(-1)^i
C^{i}_{n} P(i)$. \end{proof}

\begin{cor}\label{sum} Let $n,m$ be non-negative integers such that $n< m$. Then
 \[
 \sum_{i=0}^n \frac{(-1)^i}{m-i}C_n^iC_{m-i}^n=
\left\{\begin{array}{cl}
\frac1m & \textrm{ if }  n=0 \\
0 & \textrm{ otherwise}\,  .\\
\end{array} \right .
\]
\end{cor}

\begin{proof} Let $n>1$ and consider the polynomial
\[P(x)=\frac{1}{n!}(m-1-x)(m-2-x)\cdots(m-(n-1)-x)=\frac{1}{m-x}\cdot C_{m-x}^n\]
 of degree $n-1$.

By Lemma \ref{polynom} we obtain
\[0=\sum_{i=0}^{n}(-1)^i C^{i}_{n} P(i)= \sum_{i=0}^n \frac{(-1)^i}{m-i}C_n^iC_{m-i}^n \,.\]

For $n=0,1$, simple calculations verify the statement of the corollary.
\end{proof}

The following result shows that the analogous one established for Lie algebras \cite{Gant} is also valid for Leibniz algebras.

\begin{thm}\label{gant2}
 Let $A$ be an automorphism of a Leibniz algebra. Then there exists a unique diagonalizable automorphism $A_0$ and a unique  nilpotent derivation $T$ such that $A=A_0\exp(T)$ and $A_0 T=TA_0$.
\end{thm}

\begin{proof}
Let $L=L_{\rho_1}\oplus \dots \oplus L_{\rho_s}$ be a decomposition of a Leibniz algebra $L$ into characteristic spaces with respect to $A$.

Let us define a linear operator $A_0 \colon L\to L$ as $A_0(x)=\rho_i x$ for $x\in L_{\rho_i}$. Then $A_0$ is obviously diagonalizable and $A_0 A=AA_0$.
Notice that if $x\in L_{\rho_i}, y\in L_{\rho_j}$ then $[A_0(x),A_0(y)])=\rho_i \rho_j [x,y]$ and
by Lemma \ref{aut} we have $[x,y]\in L_{\rho_i \rho_j}$. Therefore, $A_0([x,y])=\rho_i \rho_j [x,y]$, which implies that $A_0$ is an automorphism.

Let us denote by $Q=A_0^{-1}A$. Then $A=A_0 Q$ and $A_0 Q=QA_0$. Also note that $\Spec(Q)=\{1\}$.

Consider $P=Q-I$. Obviously, $P$ is nilpotent and hence $\log Q=\log(I+P)=P-\frac12P^2+\dots+\frac{(-1)^{n-1}}{n}P^n+\cdots$ diverges.

Since $P$ is nilpotent, $\log Q$ is also a nilpotent transformation. We will prove that $\log (I+P)$ is a derivation, i.e.,
\begin{equation}\label{3.4}
\sum_{k=1}^{\infty} \frac{(-1)^{k-1}}{k}P^k([x,y])=\Big[\sum_{k=1}^{\infty} \frac{(-1)^{k-1}}{k}P^k(x),y \Big]+\Big[x,\sum_{k=1}^{\infty} \frac{(-1)^{k-1}}{k}P^k(y)\Big] \, .
\end{equation}
By Lemma \ref{P^k}, formula \eqref{3.4} is valid for $P$. Putting $C_k^i=C_k^{k-i}$ and
 substituting $r=k-i$, we obtain
$\displaystyle P^k([x,y])=\sum_{r=0}^k \sum_{j=0}^{k-r} C_k^{r} C_{k-r}^j [P^{k-j}(x),P^{r+j}(y)]$.

Now denote by
$\displaystyle B_{k,r}=\sum_{j=0}^{k-r} C_k^r C_{k-r}^j [P^{k-j}(x),P^{j+r}(y)]$ for all $0\leq r \leq k$.

Then $P^k([x,y])=B_{k,0}+B_{k,1}+\cdots+B_{k,k}$.

Therefore,
\begin{multline*}
\sum_{k=1}^{\infty} \frac{(-1)^{k-1}}{k}P^k([x,y])=\sum_{k=1}^{\infty} \frac{(-1)^{k-1}}{k}(B_{k,0}+B_{k,1}+\dots+B_{k,k})\\
= \sum_{m=0}^{\infty} \Bigg( \frac{1}{2m+1}B_{2m+1,0}-\frac{1}{2m}B_{2m,1}+\dots+\frac{(-1)^m}{m+1}B_{m+1,m}\Bigg)\\
- \sum_{m=1}^{\infty} \Bigg( \frac{1}{2m}B_{2m,0}-\frac{1}{2m-1}B_{2m-1,1}+\dots+\frac{(-1)^m}{m}B_{m,m}\Bigg)\\
= \sum_{m=0}^\infty \Bigg( \sum_{t=0}^{m}\frac{(-1)^t}{2m+1-t} B_{2m+1-t,t}\Bigg)-\sum_{m=1}^\infty \Bigg( \sum_{t=0}^{m}\frac{(-1)^t}{2m-t} B_{2m-t,t}\Bigg)\\
= \sum_{m=0}^\infty \Bigg( \sum_{t=0}^{m}\frac{(-1)^t}{2m+1-t}
\sum_{j=0}^{2m+1-2t} C_{2m+1-t}^t C_{2m+1-2t}^j [P^{2m+1-t-j}(x),P^{j+t}(y)]\Bigg)\\
- \sum_{m=1}^\infty \Bigg( \sum_{t=0}^{m}\frac{(-1)^t}{2m-t}
\sum_{j=0}^{2m-2t} C_{2m-t}^t C_{2m-2t}^j [P^{2m-t-j}(x),P^{j+t}(y)]\Bigg)
\end{multline*}
\begin{multline*}
=  \sum_{m=0}^\infty \left( \frac1{2m+1}[P^{2m+1}(x),y]+\sum_{s=1}^{2m} \Bigg(
\sum_{t=0}^{s} \frac{(-1)^t}{2m+1-t} C^{t}_{2m+1-t}
C^{s-t}_{2m+1-2t} \Bigg) \right. \\
 \left. \cdot \ [P^{2m+1-s}(x),P^s(y)]+\frac{1}{2m+1}[x,P^{2m+1}(y)]\right) -\sum_{m=1}^\infty \left( \frac{1}{2m}[P^{2m}(x),y] \right. \\
 +\left.  \sum_{s=1}^{2m-1} \Bigg(
\sum_{t=0}^{s} \frac{(-1)^t}{2m-t}
C^{t}_{2m-t} C^{s-t}_{2m-2t}
\Bigg)[P^{2m-s}(x),P^s(y)]+\frac{1}{2m}[x,P^{2m}(y)] \right).
\end{multline*}

Now since $\displaystyle C^{t}_{2m+1-t} C^{s-t}_{2m+1-2t}=C^{t}_{s} C^{s}_{2m+1-t}$ and  $\displaystyle C^{t}_{2m-t} C^{s-t}_{2m-2t}=C^{t}_{s} C^{s}_{2m-t}$ we obtain
\begin{align*}
\sum_{t=0}^{s} \frac{(-1)^t}{2m+1-t}
C^{t}_{2m+1-t} C^{s-t}_{2m+1-2t}= & \ \sum_{t=0}^{s} \frac{(-1)^t}{2m+1-t}C^{t}_{s} C^{s}_{2m+1-t} \,, \\
\sum_{t=0}^{s} \frac{(-1)^t}{2m-t}
C^{t}_{2m-t} C^{s-t}_{2m-2t}= & \ \sum_{t=0}^{s} \frac{(-1)^t}{2m-t}C^{t}_{s} C^{s}_{2m-t}\, .
\end{align*}

However, by Corollary \ref{sum} the last sums are zero for all $1\leq s\leq 2m \ (1\leq s\leq 2m-1$, respectively). Hence,
\begin{multline*}
\sum_{k=1}^{\infty} \frac{(-1)^{k-1}}{k}P^k([x,y])=
\sum_{n=1}^\infty \frac{(-1)^{n-1}}{n}\left([P^{n}(x),y]+[x,P^n(y)]\right) \\
=[\sum_{n=1}^\infty \frac{(-1)^{n-1}}{n}P^{n}(x),y]+ [x,\sum_{n=1}^\infty \frac{(-1)^{n-1}}{n}P^{n}(y)]
\end{multline*}
and \eqref{3.4} is proved.

Thus, $T=\log Q$ is a nilpotent derivation of $L$ and $A=A_0\exp(T), \ A_0T=TA_0$.
Now since $\exp(T)-I$ is nilpotent, we obtain the additive Jordan--Chevalley decomposition $A=A_0+A_0(\exp(T)-I)$ of $A$. Therefore, by Proposition \ref{Hump} $A_0$, and as consequence $T$, are determined uniquely.
\end{proof}

The following theorems generalize the results from the theory of Lie algebras \cite{Jac1} to Leibniz algebras.
\begin{thm}\label{nilpotent}
 Let $L$ be a finite-dimensional complex Leibniz algebra which admits a non-degenerated derivation. Then $L$ is a nilpotent algebra. \end{thm}

\begin{proof} Let $d$ be a non-singular derivation of a Leibniz algebra $L$ and $L=L_{\rho_1}\oplus L_{\rho_2}\oplus\dots\oplus L_{\rho_k}$ be a decomposition of $L$ into characteristic spaces with respect to $d$.

Let $\alpha,\beta \in \Spec(d)$. Then by Lemma \ref{der} we have $\displaystyle [\dots[[L_{\alpha},\underbrace{ L_{\beta}],L_{\beta}],\dots,L_{\beta}]}_{k-\textrm{times}}\subseteq L_{\alpha+k\beta}$. Since for sufficiently large $k\in \mathbb{N}$ we have $\alpha+k\beta\not \in \Spec(d)$, and by Lemma \ref{der} we obtain $[\dots[[L_{\alpha}, L_{\beta}],L_{\beta}],\dots,L_{\beta}]=0$.

Thus, for $x\in L_{\beta}$ any right multiplication operator $R_x$ is nilpotent, and due to the fact that $\alpha, \beta$ were taken arbitrarily,
 it follows that every operator from $\bigcup_{i=1}^kR(L_{\rho_i})$ is nilpotent.

Now from identity \eqref{2.1} and Lemma \ref{der} it follows that $\bigcup_{i=1}^kR(L_{\rho_i})$ is a weakly closed set of an associative algebra $R(L)$. Hence, by Theorem \ref{closed} it follows that every operator from $R(L)$ is nilpotent.

Now by Theorem \ref{Engel} we obtain the result, i.e., $L$ is nilpotent.\end{proof}

\begin{rem} The following family $L(\beta)=\langle e_1,\dots, e_n\rangle$ of characteristically nilpotent Leibniz algebras, i.e. algebras in which every derivation is nilpotent, with the following multiplication
\begin{align*}
[e_0,e_0] = & \ e_2, \qquad \qquad  [e_i,e_0]=e_{i+1} \qquad \qquad  (1\le i\le n-1) \, ,\\
[e_0,e_1]= & \ \alpha_3e_3+\alpha_4e_4+\cdots+\alpha_{n-1}e_{n-1}+\theta e_n \, ,\\
 [e_i,e_1]= & \  \alpha_3e_{i+2}+\alpha_4e_{i+3}+ \cdots +\alpha_{n+1-i}e_n  \quad (1\le i\le n-2) \, ,
\end{align*}
 where $(\alpha_3, \dots, \alpha_n,\ \theta \in \mathbb{C})$ and $\alpha_i\alpha_j\neq 0$ for some $3\leq i\neq j\leq n$,
was constructed in \cite{Omirov}. This implies that the statement of Theorem \ref{nilpotent} in the opposite direction does not hold.
\end{rem}

\begin{thm}\label{aut_prime}
 Let $L$ be a finite dimensional complex Leibniz such that it admits an automorphism of prime order with no fixed-points. Then $L$ is a nilpotent algebra.
\end{thm}

\begin{proof} Let $A$ be an automorphism of Leibniz algebra $L$ with the properties given in the statement of the theorem. Since $A$ has no fixed points then $1$ is not an eigenvalue of $A$.

Let $L=L_{\rho_1}\oplus L_{\rho_2}\oplus\dots\oplus L_{\rho_k}$ be a decomposition of $L$ into characteristic spaces with respect to $A$. From the condition that $A$ is an automorphism of prime order we obtain that the spectrum of $A$ consists of primitive $p-$th roots of the unity.  Therefore, for any $\alpha,\beta\in \Spec(A)$ there exists $k\in \mathbb{N}$ such that $\alpha \beta^k=1\not \in \Spec(A)$. Hence, by Lemma \ref{aut} we obtain
\[ [\dots[[L_{\alpha},\underbrace{ L_{\beta}],L_{\beta}],\dots,L_{\beta}]}_{k-\textrm{times}}\subseteq L_{\alpha\beta^k}=0 \, .\]

Thus, for $x\in L_{\beta}$ any right multiplication operator $R_x$  is nilpotent, and similarly as in the proof of Theorem \ref{nilpotent} we obtain that $L$ is nilpotent.
\end{proof}

Let $D$ be a derivation of a Leibniz algebra $L$ such that $D$ commutes with any inner derivation. Then $D(L)\subseteq \Ann_R(L)$.
Indeed, since $D$ commutes with any right multiplication operator we have $[D(x),y]=(R_y\circ D)(x)=(D\circ R_y)(x)=D([x,y])=[D(x),y]+[x,D(y)]$ which implies $[x,D(y)]=0$ for any $x,y\in L$. Thus, $[L,D(L)]=0$ and $D(L)\subseteq \Ann_R(L)$.

\begin{lm}\label{J+D(J)} Let $J$ be an ideal of Leibniz algebra $L$ and $D$ be a derivation given on $L$. Then $J+D(J)$ is also an ideal of $L$.
\end{lm}

\begin{proof} Since for any $x\in J, y\in L$ we have
\[
[y,D(x)]=D([x,y])-[D(x),y]\in D([J,L])+[J,L]\subseteq D(J)+J \, ,
\] and so $[L,D(J)]\subseteq D(J)+J$. Therefore, $[L,J+D(J)]\subseteq J+D(J)$.

Similarly, since for any $x\in J, y\in L$ we have
\[[D(x),y]=D([x,y])-[x,D(y)]\in D([J,L])+[J,L]\subseteq D(J)+J \, ,\]
 and so
$[D(J),L]\subseteq D(J)+J$. Therefore, $[J+D(J),L]\subseteq J+D(J)$. This implies that $J+D(J)$ is an ideal of $L$.
\end{proof}

\begin{thm}\label{D(J)} Let $J$ be the solvable radical of a Leibniz algebra $L$ and $D$ be a derivation. Then $D(J)\subseteq J$.
\end{thm}

\begin{proof} By Lemma \ref{J+D(J)} it follows that $J+D(J)$ is an ideal of Leibniz algebra $L$.
We have \[(J+D(J))^{(2)}=[J+D(J),J+D(J)]\subseteq J+[D(J),D(J)]\subseteq J+D^2(J^{(2)})\, .\]
Now assume that
\begin{equation}\label{3.5}
(J+D(J))^{(k)}\subseteq J+D^{2^{k-1}}(J^{(k)})
\end{equation}
 for some natural $k>1$.

Then
\begin{multline*}
(J+D(J))^{(k+1)}=[(J+D(J))^{(k)},(J+D(J))^{(k)}] \\ \subseteq [J+D^{2^{k-1}}(J^{(k)}),J+D^{2^{k-1}}(J^{(k)})]\subseteq
J+[D^{2^{k-1}}(J^{(k)}),D^{2^{k-1}}(J^{(k)})] \\ \subseteq J+D^{2^{k-1}+2^{k-1}}([J^{(k)},J^{(k)}])=J+D^{2^k}(J^{(k+1)}) \,.
\end{multline*}
Hence, \eqref{3.5} is verified.

Let $J^{(m)}=0$. Then $(J+D(J))^{(m)}\subseteq J+D^{2^{m-1}}(J^{(m)})=J$.
Now $(J+D(J))^{(2m-1)}=\left((J+D(J))^{(m)}\right)^{(m)}\subseteq J^{(m)}=0$.

Hence, $J+D(J)$ is a solvable ideal of Leibniz algebra $L$. Since $J$ is the solvable radical of $L$, it follows that $J+D(J)\subseteq J$ and therefore, $D(J)\subseteq J$.
\end{proof}

\begin{rem} In Theorem \ref{D(J)} if $J$ is the nilradical, analogous arguments establish the invariance of $J$ with respect to any derivation of $L$.
\end{rem}

It is not difficult to verify that a derivation in a Leibniz algebra induces a derivation in the corresponding Lie quotient algebra. However, the following example shows that the inverse is not necessarily true, i.e., not every derivation in the Lie quotient algebra can be extended to a derivation of the Leibniz algebra.
\begin{ex}
Consider a Leibniz algebra $L=\langle e_1,\dots, e_m, f_1,\dots,f_m\rangle$ with the following multiplication
\begin{align*}
[e_i,e_i]= & \ f_i, \  \ 1\leq i \leq m \,\\
[e_1,e_i]= & \ f_i, \  \ 1\leq i\leq m \, .
\end{align*}

Then $L^{\ann}=\langle f_1,\dots,f_m\rangle$ and $L/L^{\ann}$ is an abelian Lie algebra. Therefore, any linear operator in $L/L^{\ann}$ is a derivation.

Now consider an arbitrary derivation $d \colon L\to L$.

Since $[e_p,e_1]=0$ for $p>1$, we have that $0=d([e_p,e_1])=[d(e_p),e_1]+[e_p,d(e_1)]$.

If $d(e_p)=d_{1p}e_1+\dots+d_{mp}e_m+c_{1p}f_1+\dots+c_{mp}f_m$ then $[d(e_p),e_1]=d_{1p}[e_1,e_1]=d_{1p}f_1$.

Now if $d(e_1)=d_{11}e_1+\dots+d_{m1}e_m+c_{11}f_1+\dots+c_{m1}f_m$ then $[e_p,d(e_1)]=d_{p1}[e_p,e_p]=d_{p1}f_p$. Hence we obtain a condition $d_{1p}f_1+d_{p1}f_p=0$ which implies $d_{1p}=d_{p1}=0$ for all $2\leq p \leq m$. Therefore, not every derivation of $L/L^{\ann}$ can be extended to $L$.
\end{ex}

\section*{Acknowledgments}

 The  first author was supported by MICINN grant MTM 2009-14464-C02 (European
FEDER support included) and by Xunta de Galicia grant Incite 09207215 PR.

\address{\small \rm  Manuel Ladra: Department of Algebra,  University of Santiago de Compostela, 15782
Santiago de Compostela, Spain}\\ \email{manuel.ladra@usc.es}

\address{\small \rm  Ikrom  Rikhsiboev: Malaysian Institute of Industrial Technology (MITEC),
University Kuala Lumpur (UniKL),
81750 Masai,
Johor Darul Takzim, Malaysia}\\ \email{ikromr@gmail.com}

\address{\small \rm  Rustam Turdibaev: Department of Mathematics, National University of Uzbekistan,
Vuzgorogok, 27, 100174 Tashkent, Uzbekistan}\\ \email{rustamtm@yahoo.com}

\end{document}